\pdfoutput=1 

\documentclass[reqno,12pt]{amsart}

\usepackage[margin=1.3in,letterpaper,portrait]{geometry}

\usepackage[dvipsnames]{xcolor}
\usepackage{amssymb,amsmath,amsthm,amsbsy,algpseudocode,ragged2e,graphicx,framed,bm,soul,tabularx,multicol,mathtools,fancyhdr,setspace,tikz,listings, microtype,scrextend, float, chngcntr,  
etoolbox, comment,
    marvosym, 
}

\usepackage[colorlinks=true, pdfstartview=FitV, linkcolor=red!60!gray, citecolor=blue,
filecolor=violet, urlcolor=violet]{hyperref}

\usepackage{tikz-cd}
\usetikzlibrary{calc}

\usepackage{mathdots} 

\usepackage{ytableau}
\usepackage[vcentermath,enableskew]{youngtab}

\theoremstyle{definition} 
\newtheorem{theorem}{Theorem}[section]
\newtheorem{lemma}[theorem]{Lemma}

\newtheorem{proposition}[theorem]{Proposition}

\newtheorem{definition}[theorem]{Definition}
\newtheorem{example}[theorem]{Example}
\newtheorem{remark}[theorem]{Remark}
\numberwithin{equation}{section}
\counterwithout{figure}{section}

\definecolor{dblackcolor}{rgb}{0.0,0.0,0.0}
\definecolor{dbluecolor}{rgb}{0.01,0.02,0.7}
\definecolor{dgreencolor}{rgb}{0.2,0.4,0.0}
\definecolor{dgraycolor}{rgb}{0.30,0.3,0.30}

\definecolor{forGreen}{RGB}{0,120,0}

\definecolor{DarkOrange}{rgb}{0.82, 0.41, 0.12}

\title[Type $B$ $c$-Birkhoff polytopes]{Type $B$ $c$-Birkhoff polytopes\\ are order polytopes}

\author
{
Esther Banaian} 
\address{Institute for Mathematics, Paderborn University, Paderborn (Germany)}

\author{
Sunita Chepuri}
\address{Department of Mathematics and Computer Science, University of Puget Sound, Tacoma, WA (USA)}

\author{
Emily Gunawan}
\address{Department of Mathematics and Statistics, University of Massachusetts Lowell, Lowell, MA (USA)}

\author{
Jianping Pan}
\address{School of Mathematical and Statistical Sciences, Arizona State University, Tempe, AZ (USA)}


\usepackage
[sorting=nty, 
style=alphabetic,
maxbibnames=99,
backend=biber,
isbn=false,
url=false,
eprint=true 
]
{biblatex}

\allowdisplaybreaks 

\usepackage{tikz}
\usepackage{xparse}

\definecolor{forGreen}{RGB}{0,120,0}

\usepackage{circledsteps} 

\usetikzlibrary{shapes.geometric}

\newcommand{\rw}[1]{\left[{#1}\right]} 
\newcommand{\Coxeterlength}{\ell}

\DeclareMathOperator{\sort}{{\sf sort}_c}

\DeclareMathOperator{\MapPermutation}{Perm}

\newcommand{\HeapLEQ}{\preccurlyeq}
\newcommand{\HeapLessThan}{\prec}

\newcommand{\bir}{{\sf Birk}}
\newcommand{\aff}{{\sf Aff}}
\newcommand{\ord}{\mathcal{O}}
\newcommand{\heap}{{\sf Heap}}

\newcommand{\upperbarletter}{u}
\newcommand{\lowerbarletter}{d}

\newcommand{\Heapwnot}{\heap(\sort(w_0))}

\newcommand\LatticeC{{\mathcal L}({\text{c-singletons})}}

\newcommand{\X}{\color{red}\mathsf{X}}

\newcommand{\PermMatrix}[1]{X(#1)}

\def\LongEltA{{w_0^A}}
\def\LongEltB{{w_0^B}}
\def\genA{s^A}
\def\genB{s^B}
\def\cA{{c^A}}
\def\cB{{c^B}} 
\def\HHA{H^A}
\def\HHB{H^B}
\def\UnfoldMap{\eta}

\subjclass[2020]{52B20,  
05A05,  
06A07}  	
\begin{document}
\begin{abstract} 
In a previous work, we defined (type $A$) $c$-Birkhoff polytopes and showed that they were unimodularly equivalent to order polytopes of heap posets.  In this note we answer the question: What about type $B$?
\end{abstract}

\maketitle


\section{Introduction}

In this paper we build on work in \cite{cBirkhoffA} developing the connection between order polytopes and Birkhoff subpolytopes associated to reduced words of permutations.

Given a finite poset $H$, the \emph{order polytope} $\ord(H)$ is the convex hull of the indicator vectors of the order ideals of $H$.  The order polytope is a well-studied polytope in $\mathbb{R}^{\vert H \vert}$; its dimension is $\vert H \vert$ and its normalized volume is the number of linear extensions of~$H$ \cite{Sta86}.

The \emph{Birkhoff polytope} is defined as the convex polytope of $n\times n$ doubly stochastic matrices, i.e., matrices with  non-negative entries whose rows and columns all sum to~$1$.  Alternatively it can be defined as the convex hull of all permutation matrices \cite{birkhoff1946three}.  The Birkhoff polytope is notable for its wide array of applications, including in combinatorics~\cite{Igor,BPD03,A05,P13}, representation theory~\cite{BHNP09}, optimization~\cite{M88,BS96,BS03} and statistics~\cite{BRT97,DE06,PSU19}.

In~\cite{DS18}, Davis and Sagan studied the convex hull of $132$ and $312$ avoiding permutation matrices, a subpolytope of the Birkhoff polytope. They proved that the normalized volume of this polytope is the number of longest chains in the type $A_{m-1}$ Tamari lattice.
Inspired by their work and the fact that the $132$ and $312$ avoiding permutations are exactly the $c$-singletons for the Coxeter element $c=(12\dots m)$ written in cycle notation, in~\cite{cBirkhoffA} we defined a (type $A$) Birkhoff subpolytope $\bir(c)$ to be the convex hull of permutation matrices corresponding to $c$-singletons for any Coxeter element $c$. 
We then proved that $\bir(c)$ is integrally equivalent to the order polytope of the heap of the longest $c$-sorting word of $S_m$. 
A consequence of this result is that the normalized volume of  $\bir(c)$ is the number of longest (length~$\binom{m}{2}$) chains in Reading's (type $A_{m-1}$) $c$-Cambrian lattice, an important generalization of the Tamari lattice~\cite{reading2006cambrian}.

In the present paper, we turn our attention to the Coxeter group $B_{n}$ which is realized as the group of permutations $v$ on $\pm[n]=\{-n,\dots,-1,1,\dots,
n\}$ satisfying $v(-k)=-v(k)$; such permutations are called \emph{signed permutations} on $[n]$.  We can naturally embed $B_n$ into $A_{2n-1}$ by identifying each $v\in B_n$ with a permutation $\UnfoldMap(v)\in S_{2n}=A_{2n-1}$. For more details, see Section \ref{sec:embedding}.
For each Coxeter element $\cB$ in $B_n$, we use this embedding to define a (type $B$) Birkhoff subpolytope $\bir(\cB)$ to be 
the convex hull of $$\{X(\UnfoldMap(v)) \mid v \text{ is a }\cB\text{-singleton in } B_n\}$$
where $\PermMatrix{\UnfoldMap(v)}$ denotes the permutation matrix of $\UnfoldMap(v)$.

\vspace{1\baselineskip}
\noindent \textbf{Main Result} (Theorem~\ref{thm:Main})\textbf{.}
$\bir(\cB)$ is integrally equivalent to the order polytope of the heap of the longest $\cB$-sorting word of $B_n$. 
\vspace{1\baselineskip}

As in the type $A$ work~\cite{cBirkhoffA}, a consequence of this result is that the normalized volume of $\bir(\cB)$ is the number of longest (length $n^2$) chains in the (type $B_n$) $\cB$-Cambrian lattice.


\section{Background and notation}

\newcommand{\posetelt}[1]{{#1}}
\newcommand{\posetLABEL}[1]{{#1}}
\newcommand{\posetLABELwithS}[1]{s_{#1}}

A \emph{Coxeter system} $(W,S)$ is a \emph{Coxeter group} $W$ together with a set $S$ of generators for $W$ called \emph{simple reflections} subject to the relations $s^2=e$ for all $s\in S$ and the braid relations $(st)^{m(s,t)}=e$ for all $s,t$ such that $m(s,t) < \infty$. 
For $s,t\in S$ where $m(s,t)=2$, we have $st=ts$, which we call a \emph{commutation relation}.
An application of a commutation relation to a product of simple reflections is called a \emph{commutation move}. 
A \emph{Coxeter element} $c$ in $W$ is a product of all simple reflections in any order, where each reflection appears exactly once.

Given $w\in W$, the minimum number of simple reflections among all  expressions for $w$ as a product of simple reflections is called the \emph{length} of $w$, and is denoted by $\Coxeterlength(w)$. 
A \emph{reduced decomposition} of $w$ is an expression $w=s_{i_1} \cdots s_{i_{\ell(w)}}$ realizing $\Coxeterlength(w)$.

\subsection{Type $A_n$ and $B_n$ permutations} 

This paper focuses on the Coxeter groups of type $A$ and $B$, which we denote by $A_n$ and $B_n$. We now review the combinatorial realizations of these groups in terms of permutations and signed permutations. For more details, see for example~\cite{BB05}.
The simple reflections in $A_n$ are denoted $\genA_1,\dots,\genA_n$; and the simple reflections in $B_n$ are denoted  $\genB_0,\genB_1,\ldots,\genB_{n-1}$. 
We sometimes write $s_k$ when $W$ is understood.

Let $A_n$ denote the symmetric group on $n+1$ elements. 
We can represent a permutation $w\in A_n$ in \emph{one-line notation} as $w=w(1)w(2)\dots w(n+1)$. 
The simple reflections for $A_n$ are \emph{adjacent transpositions} $\genA_k=(k \quad k+1)$ for $k\in [n]$ where $[n] = \{1,2,\dots,n\}$. 
Distinct simple reflections satisfy the commutation relation $\genA_i \genA_j=\genA_j \genA_i$ if and only if $|i-j| > 1$. 
The longest element of $A_n$ is the permutation $\LongEltA=(n+1)n \dots 321$ and $\ell(\LongEltA)=\binom{n+1}{2}$.

 Let $B_n$ be the group of signed permutations on \[\pm [n]=\{ -n,\ldots ,-2,-1,1,2,\ldots,n\}\] which satisfies $w(-k)=-w(k)$ for all $k \in [n]$. We write these permutations 
in \emph{full one-line notation} as $
w(-n) w(-n+1) \dots w(-1) w(1) w(2) \dots w(n)
$
 or 
in \emph{window notation} as 
$w(1) w(2) \dots w(n).$ 
The simple reflections for $B_n$ are $\genB_0=({-1} \quad ~ 1)$ and $\genB_k = ({-k-1} \quad {-k})(k  \quad {k+1})$ for $k\in[n-1]$.
As in $A_n$, distinct simple reflections in $B_n$ satisfy the commutation relation $\genB_i \genB_j=\genB_j \genB_i$ if and only if $|i-j| > 1$. 
The longest element of $B_n$ is the signed permutation $\LongEltB = {(-1)}{(-2)} \dots {(-n)}$ in window notation and $\ell(\LongEltB)=n^2$.

To simplify notation, we refer to a reduced decomposition $s_{i_1} \cdots s_{i_{\ell(w)}}$ of $w$ in $A_n$ or $B_n$ via its \emph{reduced word} $\rw{i_1 \cdots i_{\ell(w)}}$. 
Given a reduced word $\rw{u}$, the equivalence class consisting of all words that can be obtained from $\rw{u}$ by a sequence of commutation moves is called the \emph{commutation class} of~$\rw{u}$.

\subsection{Heaps}
We begin by reviewing the classical theory of heaps; see \cite{Viennot86lecturenotes} and \cite[Solutions to Exercise 3.123(ab)]{Sta12}. For applications of heaps in voting theory, see \cite{GR08,LL20,RT25}.  We follow the exposition given in~\cite{Ste96}, where the theory of heaps was used to study fully commutative elements of a Coxeter group.

\begin{definition}
\label{defn:heap of a reduced word}
Let $W$ be the Coxeter group $A_n$ or $B_n$. 
Given a reduced word $\rw{a}=\rw{a_1 \cdots a_\ell}$ of an element in $W$, consider the
partial order $\HeapLEQ$ on the set $\{ 1, \dots, \ell\}$ obtained via the transitive closure of the relations 
\[x
\HeapLessThan
y\]
for $x < y$ such that $|a_x-a_y| \leq 1$. 
For each $1 \leq x \leq \ell$, the \emph{label} of the poset element $x$ is~$a_x$. This labeled poset is called the \emph{heap} for $\rw{a}$, denoted $\heap(\rw{a})$.

The Hasse diagram for this poset with elements $\{1,\ldots, \ell\}$ replaced by their labels is called the \emph{heap diagram} for~$\rw{a}$. 
In our figures, we represent each label $j$ by the simple reflection $s_j$ for clarity. 
\end{definition}

\def\myscale{.69}
\def\MyLocalGen{s}
\begin{figure}[htb!] 
\begin{center}
\begin{tikzpicture}[scale=\myscale]

\node (B0) at (0,0) {$3$};
\node (B1) at (1,1) {$4$};
\node (B-1) at (-1,1) {$5$};
\node (B2) at (2,2) {$6$};
\node (B-2) at (-2,2) {$7$};
\node (B3) at (3,1) {$1$};
\node (B-3) at (-3,1) {$2$};

\node (B0L2) at (0,0+2) {$10$};
\node (B1L2) at (1,1+2) {$11$};
\node (B-1L2) at (-1,1+2) {$12$};
\node (B2L2) at (2,2+2) {$13$};
\node (B-2L2) at (-2,2+2) {$14$};
\node (B3L2) at (3,1+2) {$8$};
\node (B-3L2) at (-3,1+2) {$9$};

\node (B0L3) at (0,0+2+2) {$17$};
\node (B1L3) at (1,1+2+2) {$18$};
\node (B-1L3) at (-1,1+2+2) {$19$};
\node (B2L3) at (2,2+2+2) {$20$};
\node (B-2L3) at (-2,2+2+2) {$21$};
\node (B3L3) at (3,1+2+2) {$15$};
\node (B-3L3) at (-3,1+2+2) {$16$};

\node (B0L4) at (0,0+2+2+2) {$24$};
\node (B1L4) at (1,1+2+2+2) {$25$};
\node (B-1L4) at (-1,1+2+2+2) {$26$};
\node (B2L4) at (2,2+2+2+2) {$27$};
\node (B-2L4) at (-2,2+2+2+2) {$28$};
\node (B3L4) at (3,1+2+2+2) {$22$};
\node (B-3L4) at (-3,1+2+2+2) {$23$};

\begin{scope}[black, thick, shorten <=-3pt, shorten >=-3pt]
\draw (B0) -- (B1); \draw (B1) -- (B2); \draw (B2) -- (B3);

\draw (B0) -- (B-1); \draw (B-1) -- (B-2); \draw (B-2) -- (B-3);

\draw (B0L2) -- (B1L2); \draw (B1L2) -- (B2L2); \draw (B2L2) -- (B3L2);

\draw (B0L2) -- (B-1L2); \draw (B-1L2) -- (B-2L2); \draw (B-2L2) -- (B-3L2);

\draw (B0L3) -- (B1L3); \draw (B1L3) -- (B2L3); \draw (B2L3)-- (B3L3);

\draw (B0L3) -- (B-1L3); \draw (B-1L3) -- (B-2L3); \draw (B-2L3) -- (B-3L3);

\draw (B0L4) -- (B1L4); \draw (B1L4) -- (B2L4); \draw (B2L4) -- (B3L4);

\draw (B0L4) -- (B-1L4); \draw (B-1L4) -- (B-2L4); \draw (B-2L4) -- (B-3L4);

\draw (B2) -- (B3L2); \draw (B2L2) -- (B3L3); \draw (B2L3) -- (B3L4);

\draw (B-2) -- (B-3L2); \draw (B-2L2) -- (B-3L3); \draw (B-2L3) -- (B-3L4);

\draw (B2) -- (B1L2); \draw (B2L2) -- (B1L3);  \draw (B2L3) -- (B1L4); 

\draw (B-2) -- (B-1L2); \draw (B-2L2) -- (B-1L3);  \draw (B-2L3) -- (B-1L4); 

\draw (B1) -- (B0L2); \draw (B1L2) -- (B0L3);  \draw (B1L3) -- (B0L4); 

\draw (B-1) -- (B0L2); \draw (B-1L2) -- (B0L3);  \draw (B-1L3) -- (B0L4); 
\end{scope}
\end{tikzpicture}
 \qquad
\def\MyLocalGen{s}
\begin{tikzpicture}[scale=\myscale]
\node (B0) at (0,0) {$\MyLocalGen_4$};
\node (B1) at (1,1) {$\MyLocalGen_5$};
\node (B-1) at (-1,1) {$\MyLocalGen_3$};
\node (B2) at (2,2) {$\MyLocalGen_6$};
\node (B-2) at (-2,2) {$\MyLocalGen_2$};
\node (B3) at (3,1) {$\MyLocalGen_7$};
\node (B-3) at (-3,1) {$\MyLocalGen_1$};

\node (B0L2) at (0,0+2) {$\MyLocalGen_4$};
\node (B1L2) at (1,1+2) {$\MyLocalGen_5$};
\node (B-1L2) at (-1,1+2) {$\MyLocalGen_3$};
\node (B2L2) at (2,2+2) {$\MyLocalGen_6$};
\node (B-2L2) at (-2,2+2) {$\MyLocalGen_2$};
\node (B3L2) at (3,1+2) {$\MyLocalGen_7$};
\node (B-3L2) at (-3,1+2) {$\MyLocalGen_1$};

\node (B0L3) at (0,0+2+2) {$\MyLocalGen_4$};
\node (B1L3) at (1,1+2+2) {$\MyLocalGen_5$};
\node (B-1L3) at (-1,1+2+2) {$\MyLocalGen_3$};
\node (B2L3) at (2,2+2+2) {$\MyLocalGen_6$};
\node (B-2L3) at (-2,2+2+2) {$\MyLocalGen_2$};
\node (B3L3) at (3,1+2+2) {$\MyLocalGen_7$};
\node (B-3L3) at (-3,1+2+2) {$\MyLocalGen_1$};

\node (B0L4) at (0,0+2+2+2) {$\MyLocalGen_4$};
\node (B1L4) at (1,1+2+2+2) {$\MyLocalGen_5$};
\node (B-1L4) at (-1,1+2+2+2) {$\MyLocalGen_3$};
\node (B2L4) at (2,2+2+2+2) {$\MyLocalGen_2$};
\node (B-2L4) at (-2,2+2+2+2) {$\MyLocalGen_2$};
\node (B3L4) at (3,1+2+2+2) {$\MyLocalGen_3$};
\node (B-3L4) at (-3,1+2+2+2) {$\MyLocalGen_1$};

\begin{scope}[black, thick, shorten <=-4pt, shorten >=-4pt]
\draw (B0) -- (B1); \draw (B1) -- (B2); \draw (B2) -- (B3);

\draw (B0) -- (B-1); \draw (B-1) -- (B-2); \draw (B-2) -- (B-3);

\draw (B0L2) -- (B1L2); \draw (B1L2) -- (B2L2);\draw (B2L2) -- (B3L2);

\draw (B0L2) -- (B-1L2); \draw (B-1L2) -- (B-2L2); \draw (B-2L2) -- (B-3L2);

\draw (B0L3) -- (B1L3); \draw (B1L3) -- (B2L3); \draw (B2L3) -- (B3L3);
\draw (B0L3) -- (B-1L3);  \draw (B-1L3) -- (B-2L3); \draw (B-2L3)-- (B-3L3);

\draw (B0L4) -- (B1L4);\draw (B1L4) -- (B2L4); \draw (B2L4) -- (B3L4);

\draw (B0L4) -- (B-1L4);\draw (B-1L4) -- (B-2L4); \draw (B-2L4) -- (B-3L4);

\draw (B2) -- (B3L2); \draw (B2L2) -- (B3L3); \draw (B2L3) -- (B3L4);

\draw (B-2) -- (B-3L2); \draw (B-2L2) -- (B-3L3); \draw (B-2L3) -- (B-3L4);

\draw (B2) -- (B1L2); \draw (B2L2) -- (B1L3);  \draw (B2L3) -- (B1L4); 

\draw (B-2) -- (B-1L2); \draw (B-2L2) -- (B-1L3);  \draw (B-2L3) -- (B-1L4); 

\draw (B1) -- (B0L2); \draw (B1L2) -- (B0L3);  \draw (B1L3) -- (B0L4); 

\draw (B-1) -- (B0L2); \draw (B-1L2) -- (B0L3);  \draw (B-1L3) -- (B0L4); 
\end{scope}
\end{tikzpicture}
\bigskip

\begin{tikzpicture}[scale=\myscale]
\node (B0) at (0,0) {$2$};
\node (B1) at (1,1) {$3$};
\node (B2) at (2,2) {$4$};
\node (B3) at (3,1) {$1$};

\node (B0L2) at (0,0+2) {$6$};
\node (B1L2) at (1,1+2) {$7$};
\node (B2L2) at (2,2+2) {$8$};
\node (B3L2) at (3,1+2) {$5$};

\node (B0L3) at (0,0+2+2) {$10$};
\node (B1L3) at (1,1+2+2) {$11$};
\node (B2L3) at (2,2+2+2) {$12$};
\node (B3L3) at (3,1+2+2) {$9$};

\node (B0L4) at (0,0+2+2+2) {$14$};
\node (B1L4) at (1,1+2+2+2) {$15$};
\node (B2L4) at (2,2+2+2+2) {$16$};
\node (B3L4) at (3,1+2+2+2) {$13$};

\begin{scope}[black, thick, shorten <=-3pt, shorten >=-3pt]
\draw (B0) -- (B1); \draw (B1)-- (B2); \draw (B2)-- (B3);

\draw (B0L2) -- (B1L2); \draw (B1L2)-- (B2L2); \draw (B2L2)-- (B3L2);

\draw (B0L3) -- (B1L3); \draw (B1L3)-- (B2L3); \draw (B2L3)-- (B3L3);

\draw (B0L4) -- (B1L4); \draw (B1L4)-- (B2L4); \draw (B2L4)-- (B3L4);

\draw (B2) -- (B3L2); \draw (B2L2) -- (B3L3); \draw (B2L3) -- (B3L4);

\draw (B2) -- (B1L2); \draw (B2L2) -- (B1L3);  \draw (B2L3) -- (B1L4); 

\draw (B1) -- (B0L2); \draw (B1L2) -- (B0L3);  \draw (B1L3) -- (B0L4); 
\end{scope}
\end{tikzpicture}
 \qquad\qquad\qquad
\def\MyLocalGen{s}
\begin{tikzpicture}[scale=\myscale]
\node (B0) at (0,0) {$\MyLocalGen_0$};
\node (B1) at (1,1) {$\MyLocalGen_1$}; 
\node (B2) at (2,2) {$\MyLocalGen_2$}; 
\node (B3) at (3,1) {$\MyLocalGen_3$}; 

\node (B0L2) at (0,0+2) {$\MyLocalGen_0$};
\node (B1L2) at (1,1+2) {$\MyLocalGen_1$}; 
\node (B2L2) at (2,2+2) {$\MyLocalGen_2$}; 
\node (B3L2) at (3,1+2) {$\MyLocalGen_3$}; 

\node (B0L3) at (0,0+2+2) {$\MyLocalGen_0$};
\node (B1L3) at (1,1+2+2) {$\MyLocalGen_1$}; 
\node (B2L3) at (2,2+2+2) {$\MyLocalGen_2$};
\node (B3L3) at (3,1+2+2) {$\MyLocalGen_3$};

\node (B0L4) at (0,0+2+2+2) {$\MyLocalGen_0$};
\node (B1L4) at (1,1+2+2+2) {$\MyLocalGen_1$}; 
\node (B2L4) at (2,2+2+2+2) {$\MyLocalGen_2$}; 
\node (B3L4) at (3,1+2+2+2) {$\MyLocalGen_3$};

\begin{scope}[black, thick, shorten <=-4pt, shorten >=-4pt]
\draw (B0) -- (B1); \draw (B1) -- (B2); \draw (B2) -- (B3);

\draw (B0L2) -- (B1L2); \draw (B1L2)-- (B2L2); \draw (B2L2)-- (B3L2);

\draw (B0L3) -- (B1L3); \draw (B1L3)-- (B2L3); \draw (B2L3)-- (B3L3);

\draw (B0L4) -- (B1L4); \draw (B1L4) -- (B2L4); \draw (B2L4) -- (B3L4);

\draw (B2) -- (B3L2); \draw (B2L2) -- (B3L3); \draw (B2L3) -- (B3L4);

\draw (B2) -- (B1L2); \draw (B2L2) -- (B1L3);  \draw (B2L3) -- (B1L4);

\draw (B1) -- (B0L2); \draw (B1L2) -- (B0L3);  \draw (B1L3) -- (B0L4); 
\end{scope}
\end{tikzpicture}
\end{center}
\caption{
Top (left):
Hasse diagram of the underlying poset of $\heap(\rw{aaaa})$ for $\rw{a}=\rw{71 4 53 62}$ in $A_7$;
top (right): heap diagram of $\heap(\rw{aaaa})$. 
Bottom (left): 
Hasse diagram of the underlying poset of $\heap(\rw{bbbb})$ for $\rw{b}=\rw{3012}$ in $B_4$; bottom (right): 
heap diagram of $\heap(\rw{bbbb})$}
\label{fig:longest element heap A7 and B4}
\end{figure}
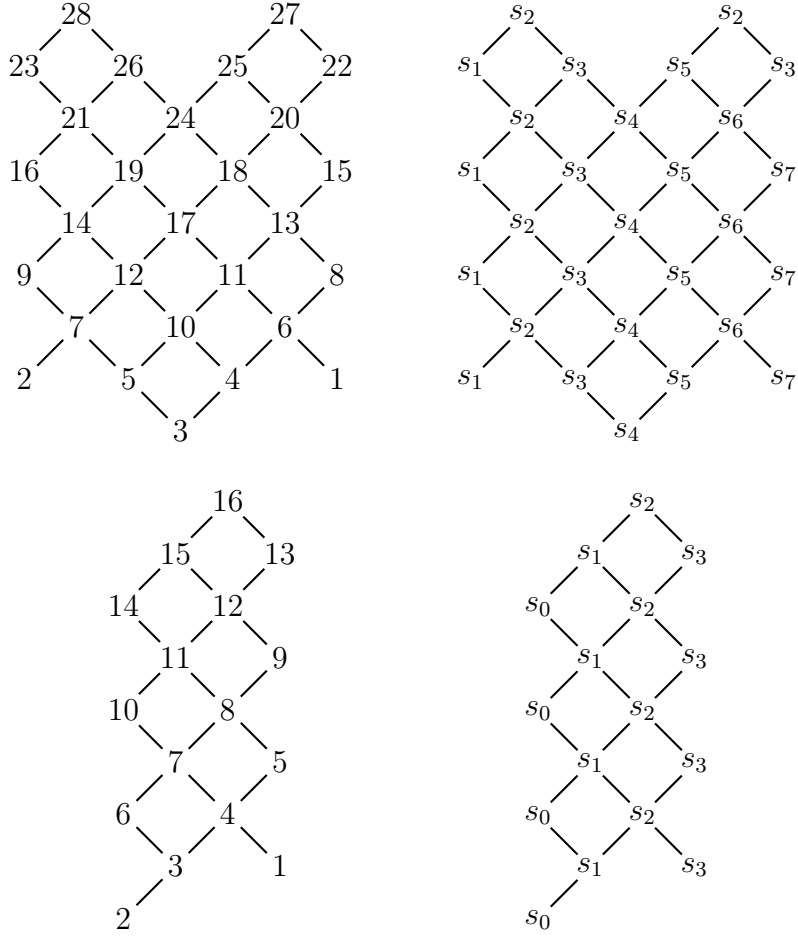

\begin{example}
\label{ex:running example type B}
\begin{enumerate}
\item 
The top two pictures in Figure~\ref{fig:longest element heap A7 and B4} show the Hasse diagram and heap diagram of $\heap(\rw{aaaa})$ for $\rw{a}=\rw{71 4 53 62}$. The elements of the underlying poset are $\{1,2,\dots, 28 \}$, and the possible labels are $\{1,2,\dots,7\}$.
\item 
The bottom two pictures in Figure~\ref{fig:longest element heap A7 and B4} show the Hasse diagram and heap diagram of  $\heap(\rw{bbbb})$ for $\rw{b}=\rw{3012}$. The elements of the underlying poset are $\{1,2,\dots, 16 \}$, and the possible labels are $\{0,1,2,3\}$.
\end{enumerate}
\end{example}


We can understand the commutation class of $\rw{a}$ by looking at linear extensions of $\heap(\rw{a})$.

\begin{definition}[linear extension]
Let $(H,\HeapLEQ_H)$ be a poset on $\ell$ elements. A \emph{linear extension} of $H$ is a bijection $\pi:\{1,\dots, \ell\} \to H$
that is consistent with the structure of the poset, that is,  
\begin{center}
$\pi(x) \HeapLessThan_H \pi(y)$ implies $x < y$.
 \end{center}
\end{definition}

If $\rw{a}=\rw{a_1 \cdots a_\ell}$ is a reduced word, then the poset $\heap(\rw{a})$ is a poset on $\{1,\dots,\ell\}$, and so we can think of a linear extension of 
$\heap(\rw{a})$ as a permutation
\[\pi = \begin{pmatrix}
1 & 2 & \dots & \ell
\\
\pi(1) & \pi(2) & \dots &\pi(\ell)
\end{pmatrix}.\]  This leads us to the next definition.

\begin{definition}[Labeled linear extension]
For a reduced word $\rw{a}=\rw{a_1 \cdots a_\ell}$, a \emph{labeled linear extension} of 
$\heap(\rw{a})$ is a word $\rw{a_{\pi(1)} \cdots a_{\pi(\ell)}}$ where $\pi$ is a linear extension of $\heap(\rw{a})$.
\end{definition}

\begin{proposition}
{\cite[Proof of Proposition~2.2]{Ste96}}
\label{prop:set of labeled linear extensions is commutativity class Stembridge}
Given a reduced word $\rw{a}$, the set of labeled linear extensions of $\heap(\rw{a})$ is the commutation class of $\rw{a}$.
\end{proposition}

\subsection{The heap of the longest $c$-sorting word in  $A_n$ and $B_n$}

The notion of $c$-sorting words  was introduced by Reading in~\cite{reading07-clusters-paper4}. 
Fix a reduced word $\rw{a_1 a_2 \dots a_n}$ for a Coxeter element $c$, and define an infinite word
\[
c^{\infty} \coloneqq a_1 a_2 \dots a_n \mid a_1 a_2 \dots a_n\mid \cdots
\]
The \emph{$c$-sorting word} of $w\in W$ is the lexicographically first (as a sequence of positions in $c^\infty$) subword of $c^\infty$ that is a reduced word for $w$. 
Denote this word by $\sort(w)$. 
If a word $\rw{u}=\rw{u_1 \dots u_\ell}$ is the $c$-sorting word of an element in $W$, we refer to $\rw{u}$ as a \emph{$c$-sorting word}.

In this paper, we are  interested in the heap diagram of $\sort(w_0)$ for $A_n$ and $B_n$. The following is proven in Sections 6.2 and 6.3 of \cite{DefantLi23}, for types $A$ and $B$, respectively. 

\begin{lemma} 
\begin{enumerate}
\item 
The $c$-sorting word for $w_0$ in $A_n$ is  a concatenation of nonempty subwords of $c$,
$
\sort(w_0)= \rw{K_1 \mid K_2  \mid \dots \mid K_p} 
$
where $K_1 \supseteq K_2 \supseteq \cdots \supseteq K_p$ as sets. 
For a construction of the heap diagram of $\sort(w_0)$ in type $A$, see for example~\cite[Algorithm~6.1]{cBirkhoffA}.

\item 
The $c$-sorting word of $w_0$ in $B_n$ is $c^n$. So, to draw the heap of $c^n$, we simply stack $n$ ``layers'' of $\heap(c)$. 
\end{enumerate}
\end{lemma}

See Figure~\ref{fig:longest element heap A7 and B4} for the heap diagrams of $\sort(w_0)$ for $c=\rw{71 4 53 62}$ in $A_7$ and for $c=\rw{3012}$ in $B_4$.

\subsection{$c$-singleton permutations}
\label{sec:c-singleton}

In \cite{HLT11}, Hohlweg, Lange, and Thomas 
introduced the notion of $c$-singletons. Following the survey~\cite{HohlwegSurvey}, we will adopt the definition that $w \in W$ is a $c$-singleton if and only if some reduced word of $w$ is a prefix of a word in the commutation class of $\sort(w_0)$.
We will also use the following  characterization of $c$-singletons, which follows from Proposition~\ref{prop:set of labeled linear extensions is commutativity class Stembridge}.

\begin{lemma}
\label{lem:c-singleton iff linear extension of order ideal}
An element $w \in W$ is a $c$-singleton if and only if there exists a reduced word~$\rw{u}$ 
of $w$ and an order ideal $I$ of $\Heapwnot$ such that $I=\heap(\rw{u})$. 
\end{lemma}

The $c$-singletons form a distributive sublattice of the right weak order on $W$, denoted $\LatticeC$ \cite{HLT11}.
For a poset $H$, let $J(H)$ denote the lattice of order ideals of $H$.   
The following results were proven in
\cite[Proposition 3]{LL20} and 
\cite[Section 2]{cBirkhoffA} for type $A$; the same proofs hold for type $B$.

\begin{proposition}
\label{prop:poset isomorphism}
Let $\rw{u}=\rw{u_1 u_2 \dots u_{\Coxeterlength(w_0)}}$ denote $\sort(w_0)$, and consider the labeled poset $H=\heap(\rw{u})$ on $\{1,2,\dots,\Coxeterlength(w_0) \}$, following Definition~\ref{defn:heap of a reduced word}. 
Given an order ideal $I$ of $H$, let $\rw{u}_I=\rw{u_i}_{i\in I}$ denote the subword of $\rw{u}$ at positions $I$.
\begin{enumerate}
\item 
The word $\rw{u}_I$ is a $c$-sorting word. 
\item 
The map 
\begin{align*}
\MapPermutation \colon 
J(H)
& \to  
\LatticeC 
\\
\nonumber I &\mapsto \rw{u}_I 
\end{align*}
is a poset isomorphism, and 
the inverse map of $\MapPermutation$ is 
\begin{align*}
f: 
\LatticeC 
& \to J(H) 
\\
w & \mapsto \heap(\sort(w))
\end{align*}
\end{enumerate}
\end{proposition}

A \emph{barring} of a set $Z$ of integers is a partition of $Z$ into two sets $\underline{Z}$ and $\overline{Z}$. 
If $\lowerbarletter \in \underline{Z}$ (resp. $\upperbarletter \in \overline{Z}$), 
we call $\lowerbarletter$ a \emph{lower-barred number} (resp. we call $\upperbarletter$ an \emph{upper-barred number}) and sometimes emphasize this by writing $\underline{\lowerbarletter}$ (resp. $\overline{\upperbarletter}$).

Let $c$ be a Coxeter element in $W$. Let $Z=[2,n]= \{2,3,\ldots,n\}$ if $W=A_n$ and $Z=\pm[1,{n-1}]=\{{-(n-1)},\dots,{-2},{-1},1,2,\ldots,{n-1} \}$ if $W=B_n$. 
First, the barring of $[2,n]$ and $[1,{n-1}]$ in types $A$ and $B$ respectively is defined as follows: 
if $s_i$ appears after $s_{i-1}$ in any reduced decomposition of $c$, then $i$ is a lower-barred number; otherwise $i$ is an upper-barred number. 
For $W=B_n$, the barring of $[1,{n-1}]$ is extended to a barring of $Z=\pm[1,{n-1}]$ by specifying that the barring of ${-i}$ is opposite the barring of $i$.

We denote the lower-barred numbers by $\lowerbarletter_1< \ldots < \lowerbarletter_r$ and 
the upper-barred numbers by $\upperbarletter_1 < \ldots < \upperbarletter_s$.

\begin{remark}[{\cite[Section 3]{reading07-clusters-paper4}}]
\label{rem:Coxeter element in cycle notation} 
We can write a Coxeter element $c$ of $A_n$ as a single cycle of length $n+1$ of the form 
\[c=
(1~  \underline{\lowerbarletter_1} ~ \dots~  \underline{\lowerbarletter_r} ~ {(n+1)} ~ \overline{\upperbarletter_s} ~ \dots ~ \overline{\upperbarletter_1})
\] 
and the Coxeter element $c$ of $B_n$ 
as a single cycle of length $2n$ of the form \[c=
({-n}~  \underline{\lowerbarletter_1} ~ \dots ~  \underline{\lowerbarletter_r} ~ n ~ \overline{\upperbarletter_s} ~ \dots ~ \overline{\upperbarletter_1}).
\] 
\end{remark}

\begin{example}\label{ex:Coxeter elements A7 and B4}
Consider the Coxeter elements
$c^A=\genA_7 \genA_1 \genA_4 \genA_5 \genA_3 \genA_6 \genA_2$ in $A_7$ and 
$c^B=\genB_3 \genB_0 \genB_1 \genB_2$ in $B_4$ as in Figure~\ref{fig:longest element heap A7 and B4}.
Their heap diagrams are as follows. 
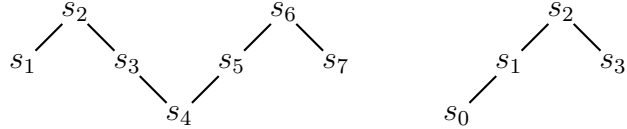
\def\myscale{0.69}
\begin{center}
\def\MyLocalGen{s}
\begin{tikzpicture}[scale=\myscale]
\node (B0) at (0,0) {$\MyLocalGen_4$};
\node (B1) at (1,1) {$\MyLocalGen_5$};
\node (B-1) at (-1,1) {$\MyLocalGen_3$};
\node (B2) at (2,2) {$\MyLocalGen_6$};
\node (B-2) at (-2,2) {$\MyLocalGen_2$};
\node (B3) at (3,1) {$\MyLocalGen_7$};
\node (B-3) at (-3,1) {$\MyLocalGen_1$};

\begin{scope}[black, thick, shorten <=-4pt, shorten >=-4pt]
\draw (B0) -- (B1); \draw (B1) -- (B2); \draw (B2) -- (B3);

\draw (B0) -- (B-1); \draw (B-1) -- (B-2); \draw (B-2) -- (B-3);
\end{scope}
\end{tikzpicture}
\qquad 
\begin{tikzpicture}[scale=\myscale]
\node (0) at (-1,0) {$\posetLABELwithS{0}$}; 

\node (1) at (0,1) {$\posetLABELwithS{1}$}; 
\node (2) at (1,2) {$\posetLABELwithS{2}$}; 
\node (3) at (2,1) {$\posetLABELwithS{3}$}; 

\begin{scope}[black, thick, shorten <=-4pt, shorten >=-4pt]
\draw (0) -- (1);
\draw (1) -- (2);
\draw (3) -- (2);
\end{scope}
\end{tikzpicture}
\end{center}
Then
$c^A=
\, (1
\,\, \underline{2} 
\,\, \underline{5} \,\, \underline{6} \,\, 8 \, \, \overline{7} \,\, 
\overline{4} \,\, 
\overline{3} \,\, 
)$ and  
$c^B=
\, (-4
\,\, \underline{-3} 
\,\, \underline{1} \,\, \underline{2} \,\, 4 \, \, \overline{3} \,\, 
\overline{-1} \,\, 
\overline{-2} \,\, 
)$. 
\end{example}

\section{Type $B$ $c$-Birkhoff polytopes}

\subsection{Embedding $B_n$ into $A_{2n-1}$}\label{sec:embedding}

For the rest of the paper, we will consider the Coxeter groups $B_n$ and $A_{2n-1}$. 
The two groups are related by an ``unfolding'' injective homomorphism $\UnfoldMap: B_n \to A_{2n-1}$, determined by \[
\UnfoldMap(\genB_i) = \begin{cases} \genA_n  & \text{if } i  = 0 \\
\genA_{n+i} \genA_{n-i} & \text{if }  i > 0\\
\end{cases}
\]
for each simple reflection $\genB_i$ of $B_n$. 

Thinking of an element $v\in B_n$ as a bijection on the set $\{{-n},\dots,{-1},1,\dots,n\}$, $\UnfoldMap(v)$ is the bijection on the set $\{1,\dots,{2n}\}$ given by $\theta \circ v \circ \theta^{-1}$ where $\theta$ maps $\{{-n},\dots,{-1},1,\dots,n\}$ to $\{1,\dots,{2n}\}$ in order.
In other words, $\UnfoldMap$ takes an element of $B_n$ in full one-line notation and replaces the numbers $-n,-n+1,\dots,-1,1,\dots,n$ with the numbers $1,2,\dots,n,n+1,\dots,2n$ in order.

Let $\cB$ denote a Coxeter element of $B_n$ and let $\cA = \UnfoldMap(\cB)$; note that $\cA$ is a Coxeter element of $A_{2n-1}$.

\begin{remark}
\label{rem:permutation from type B to type A}
The homomorphism $\UnfoldMap$ can be alternatively defined  as follows.
Let $v\in B_n$, and let $w=\UnfoldMap(v) \in A_{2n-1}=S_{2n}$.
Then, for $k=1,\dots,n$, we have 
\begin{align*}
w(n+k)&=
    v(k)+ n  
\text{ and }   
w(n+1-k)=
    -v(k)+ n + 1 &\text{ if } v(k) >0
    \\
w(n+k)&=    
    v(k)+ n + 1 
    \text{ and }
    w(n+1-k) =  -v(k)+ n
    &\text{ if } v(k) <0  
\end{align*}
\end{remark}

For example, if $v=
2 \, \, {(-3)} \, \, 4 \, \, 1 \, \, {(-1)} \, \, {(-4)} \, \, 3 \, \, {(-2)} \in B_4$ in full one-line notation then, replacing $\{-n,\dots,{-1},1,\dots,n\}$ with $\{1,\dots,2n\}$ in order, $\UnfoldMap(v)=6 \, 2 \, 8 \, 5 \, 4 \, 1 \, 7 \, 3 \in A_7$. We can also compute this using Remark~\ref{rem:permutation from type B to type A}. 
A second example is $\UnfoldMap(c^B)=c^A$, where $c^A$ and $c^B$ are as in Example~\ref{ex:Coxeter elements A7 and B4}.

\begin{remark}\label{rem:eta induces}
The unfolding homomorphism $\UnfoldMap$ induces a map that sends a reduced word $\rw{u}$ of an element in $B_n$ to a type $A$ reduced word $\eta(\rw{u})$ by replacing each copy of $0$ in $\rw{u}$ with $n$ and each copy of positive $i$ with the two-letter contiguous subword ${(n+i)} ~ {(n-i)}$.
\end{remark}

The symmetry we see in the barring of $[2,7]$ in Example~\ref{ex:Coxeter elements A7 and B4} holds in general:

\begin{lemma}\label{lem:SymmetryInBarring}
In the barring of $[2,2n-1]$ associated to the Coxeter element $\UnfoldMap(\cB) \in A_{2n-1}$,  the integer $i$ is lower-barred if and only if ${2n+1-i}$ is upper-barred. 
\begin{proof}
By the definition of $\eta$, for $k>0$ the integers ${n+k}$ and ${n-k}$ appear together in a reduced word of $\eta(\cB)$.
Thus the integer ${n+k}$ appears after ${n+k-1}$ if and only if ${n-k}$ appears after ${n-k+1}$ in any reduced word of $\UnfoldMap(\cB)$.  That is, the integer $n+k$ appears after $n+k-1$ if and only if $n-k+1$ appears before $n-k$ in any reduced word of $\UnfoldMap(\cB)$.
\end{proof}
\end{lemma} 

Given $w\in A_{2n-1}$, let $\PermMatrix{w}$ denote the corresponding permutation matrix. Specifically, let $\PermMatrix{w}$ be the matrix with $1$'s in row $i$ and column $w(i)$ for all $1 \leq i \leq 2n$ and $0$'s everywhere else. 
The $\cA$-Birkhoff polytope, denoted $\bir(\cA)$, 
was defined in \cite{cBirkhoffA}
to be the convex hull of 
 \[\{\PermMatrix{w}  \mid w  \text{ is a }\cA\text{-singleton}\}.\] 
The vertices of $\bir(\cA)$ are precisely the permutation matrices $\PermMatrix{w}$ where $w$ is a $\cA$-singleton. 
It was shown in \cite{cBirkhoffA} that $\bir(\cA)$ is integrally equivalent to the order polytope of $\heap(\mathrm{sort}_{\cA}(\LongEltA))$. 

\begin{definition}
We define the type $B_n$ $\cB$-Birkhoff polytope, denoted $\bir(\cB)$,
to be the convex hull of 
 \[\{\PermMatrix{\UnfoldMap(v)}  \mid v  \in B_n \text{ is a }\cB\text{-singleton}\}.\] 
\end{definition} 
The vertices of $\bir(\cB)$ are precisely the permutation matrices $\PermMatrix{\UnfoldMap(v)}$ where $v$ is a $\cB$-singleton. Let $\aff(\cA)$ (resp. $\aff(\cB)$) denote the affine hull of the vertices of $\bir(\cA)$
(resp. $\bir(\cB)$).

\subsection{Reflection-invariant order ideals and rotation-invariant matrices}

For the rest of this paper, let $\HHB=\heap(\mathrm{sort}_{\cB}(\LongEltB))$ and $\HHA=\heap(\mathrm{sort}_{\cA}(\LongEltA))$.

\begin{lemma}\label{lem:HA-symmetry}
Let $\rw{u}$ be a reduced word of an element in $B_n$. 
The type-$A$ heap 
$\heap(\eta(\rw{u}))$ has reflectional symmetry with respect to the vertical $y$-axis, and the right side of $\heap(\eta(\rw{u}))$ has the same underlying poset as $\heap(\rw{u})$.

In particular, $\HHA$ has reflectional symmetry with respect to the vertical $y$-axis, and the right side of $\HHA$ has the same underlying poset as $\HHB$.
\end{lemma}

\begin{proof}
The conclusion follows from the construction of the induced map $\eta$ from reduced words of $v$ to reduced words of $\eta(v)$.
\end{proof}

See
Figure~\ref{fig:longest element heap A7 and B4} for the heap diagrams of $\HHA$ and $\HHB$ where
$\cB=\genB_3\genB_0\genB_1\genB_2$and $\cA=\UnfoldMap(\cB)=
\genA_7\genA_1 
\genA_4 
\genA_5\genA_3 
\genA_6 \genA_2$. 
This example illustrates Lemma~\ref{lem:HA-symmetry}.

Let $\rho: \HHA \to \HHA$ be the ``reflection" map  which sends every vertex in the heap diagram of $\HHA$ to its reflection with respect to the $y$-axis. 
This induces a map $J(\HHA)\to J(\HHA)$ which we will also call $\rho$. 
We say that $I\in J(\HHA)$ is \emph{reflection-invariant}  if $\rho(I) = I$.  
Let $J(\HHA)^F$ denote the set of reflection-invariant order ideals in $J(\HHA)$.

By Lemma~\ref{lem:HA-symmetry}, we can view $\HHB$ as the subposet on the right side of $\HHA$.  This allows us to define a bijection $\alpha: J(\HHB) \to J(\HHA)^F$ as follows: if $I^v$ is an order ideal of $\HHB$, let $\alpha(I^v)$ be the order ideal of $\HHA$ that contains elements on the right side of $\HHA$ corresponding to the elements of $I^v$ as well as the image of these elements under the map~$\rho$.

Consider the poset isomorphisms, $\MapPermutation$ and $f$, defined in Proposition~\ref{prop:poset isomorphism}. 
To distinguish between type $A$ and type $B$ versions of these maps, we will use a superscript $A$ or $B$.

\begin{lemma}
\label{lem:UnfoldMap}
Let $v$ be a $\cB$-singleton, and let $I^v$ denote the order ideal $f^B(v)$ in $J(\HHB)$.
Then we have the following.
\begin{enumerate}
\item  \label{lem:UnfoldMap:itm:0}
$\UnfoldMap(v)=\MapPermutation^A(\alpha(I^v))$. 
\item \label{lem:UnfoldMap:itm:1}
 $\UnfoldMap(v)$ is a $\cA$-singleton.

\item \label{lem:UnfoldMap:itm:2}
The order ideal $f^A(\UnfoldMap(v))$ is reflection-invariant.
\end{enumerate}

\end{lemma}
\begin{proof} 
\eqref{lem:UnfoldMap:itm:0}
This follows from the construction of $\UnfoldMap$ and $\alpha$.

\eqref{lem:UnfoldMap:itm:1} 
Since $\UnfoldMap(v)$ is in the codomain of $\MapPermutation^A$ by \eqref{lem:UnfoldMap:itm:0}, we can conclude by Proposition~\ref{prop:poset isomorphism} that $\UnfoldMap(v)$ is a $\cA$-singleton.

\eqref{lem:UnfoldMap:itm:2} 
We have 
$\UnfoldMap(v)=\MapPermutation^A(\alpha(I^v))$ by \eqref{lem:UnfoldMap:itm:0}.  
Since $\alpha(I^v)$ is reflection-invariant and since $f^A$ and $\MapPermutation^A$ are inverse maps, the claim follows.
\end{proof}

\begin{remark}\label{rem:type B is affine subspace of type A}
It follows from Lemma~\ref{lem:UnfoldMap}\eqref{lem:UnfoldMap:itm:1} that $\bir(\cB)$ is a subpolytope of the $\cA$-Birkhoff polytope $\bir(\cA)$ and that $\aff(\cB)$ is an affine subspace of $\aff(\cA)$. 
\end{remark}

\def\wReverse{w^\text{rev}}
\def\wRevComp{w^\text{revcomp}}
\def\wComp{w^\text{comp}}

\begin{definition}
Let $w \in A_{m-1}=S_{m}$.
The \emph{reverse} of $w$, denoted $\wReverse$, is the result of writing $w$ in one-line notation backwards; that is, $\wReverse(i)=w(m+1-i)$. 
The \emph{complement} of $w$, denoted $\wComp$, is the result of replacing every entry $i$ in the one-line notation of $w$ with $m+1-i$. 
The \emph{reverse-complement} of $w$, denoted $\wRevComp$,  is the result of taking the complement of the reverse of $w$; that is, $\wRevComp(i)=m+1-w(m+1-i)$.
\end{definition}

The permutation matrix $X(\wReverse)$ is the result of reflecting the permutation matrix $X(w)$ with respect to a horizontal line, while 
$X(\wComp)$ is the result of reflecting $X(w)$ with respect to a vertical line. 
The composition of these two actions is the 180 degree rotation, and thus 
the 180 degree rotation of the permutation matrix $X(w)$ is the permutation matrix $X(\wRevComp)$ of the reverse-complement of $w$.

As a consequence, given a permutation $w \in A_{2n-1}=S_{2n}$, its permutation matrix $\PermMatrix{w}$ is invariant under 180 degree rotation if and only if $w$ is equal to its reverse-complement, that is, for all $1 \leq i \leq 2n$,  $w(i) = \wRevComp(i) = 2n+1 - w(2n+1-i)$. Equivalently, $X(w)$ is invariant under 180 degree rotation if and only if $w$ satisfies
\begin{equation} 
\label{rem:eq:180degRotation}
w({n+k}) + w({n+1-k}) = 2n+1
\end{equation} 
for all $k=1,\dots,n$.

Let $A_{2n-1}^{180}$ denote $\{ w \in A_{2n-1} \mid w \text{ satisfies \eqref{rem:eq:180degRotation}} \}$, that is, $A_{2n-1}^{180}$ is the set of permutations in $A_{2n-1}$ whose permutation matrices are invariant under 180 degree rotation. 

\begin{lemma}\label{lem:X(w) is 180 rotation invariant if w is in image of iota} 
The image  $\UnfoldMap(B_n)$ is equal to $A_{2n-1}^{180}$. 
\end{lemma}
\begin{proof} 
It follows from Remark~\ref{rem:permutation from type B to type A} 
that $\UnfoldMap(v)$ 
satisfies  \eqref{rem:eq:180degRotation} for each $v$ in~$B_n$. Thus $\UnfoldMap(B_n)$ is a subset of $A_{2n-1}^{180}$.

Conversely, we show that, if $w\in A_{2n-1}$ and $\PermMatrix{w}$ is invariant under 180 degree rotation, then $w \in \UnfoldMap(B_n)$. The inverse map $j: A_{2n-1}^{180} \to B_n$ of $\UnfoldMap$ can be described as follows. Let $w \in A_{2n-1}^{180}$.  Then $w(n+k) + w(n+1-k) = 2n+1$ for all $k=1,\dots,n$. 
Define $v'\in B_n$ to be 
\begin{align*}
v'(k)&=
\begin{cases}
    w(n+k)-n & \text{ if } w(n+k) \geq n+1 \\
     w(n+k)-n-1 & \text{ if } w(n+k) \leq n,
\end{cases}\\
 v'({-k})&=-v'(k)
\end{align*}
for all $k\in[n]$, and set $j(w)$ to be $v'$. Then $\UnfoldMap \circ j$ is the identity function on $A_{2n-1}^{180}$. 
\end{proof}

\subsection{Type $B$ lattice-preserving projection}

In \cite[Section 5]{cBirkhoffA}, we defined a projection 
$\Pi_{\cA}$
from the space of $(2n)\times(2n)$ $\mathbb{R}$-valued matrices to 
$\mathbb{R}^{\binom{2n}{2}}$ by choosing exactly $\binom{2n}{2}$ positions from a matrix $X$; the positions are determined by the Coxeter element $\cA$. 
We proved in \cite[Theorem~5.9]{cBirkhoffA} that $\Pi_{\cA}$ is injective on $\aff(\cA)$ and sends integral points to integral points.

\begin{proposition}[Zero relations {\cite[Proposition~4.4]{cBirkhoffA}}]
\label{prop:ZeroRelationsOnMatrix}
If $X \in \aff(\cA)$, then $X$ satisfies the following.
 \begin{itemize} 
 \item For each upper-barred $u$, we have $X(i,u) = 0$ for all $1 \leq i \leq \min({u-1}, {n+1-u})$.
 \item For each lower-barred $d$, we have $X(i,d) = 0$ for all $\max({d+1}, {n+3-d}) \leq i \leq {n+1}$.
\end{itemize}
\end{proposition}

The projection $\Pi_{\cA}$ never includes positions on the main diagonal or positions whose entries are guaranteed to be zero for $X \in \aff(\cA)$ by Proposition~\ref{prop:ZeroRelationsOnMatrix}.  When positions below the main diagonal are included, these positions must come from the bottom half of~$X$.

In the following,  we view the projection $\Pi_{\cA}$ as a subset of $[2n] \times [2n]$.

\begin{lemma}\label{lem:n squared}
Let $X$ be a $(2n)\times(2n)$ $\mathbb{R}$-valued matrix. 
The number of positions 
that the map $\Pi_{\cA}$ chooses from the top half (rows 1 through $n$) of $X$ is $n^2$.
\end{lemma}

\begin{proof}
The subset of entries taken by $\Pi_{\cA}$ in the top half of $X$
are exactly those above the main diagonal and not in a spot which is guaranteed to be $0$ for a matrix in $\aff(\cA)$ (as described in Proposition~\ref{prop:ZeroRelationsOnMatrix}). 
The number of entries above the main diagonal and in the top half are $(2n-1) + (2n-2) + \cdots + n = n^2 + {\binom{n}{2}}$.  
The total number of entries in $X$ guaranteed to be $0$ is $2 {\binom{n}{2}}$, and Lemma~\ref{lem:SymmetryInBarring} implies that exactly half of these will be above the main diagonal.  Therefore,  $\Pi_{\cA}$ restricted to $[n] \times [2n]$ has exactly $n^2$ entries.  
\end{proof}

It follows from Lemma~\ref{lem:X(w) is 180 rotation invariant if w is in image of iota} that, for each $v \in B_n$, the top half of $\PermMatrix{\UnfoldMap(v)}$ determines its bottom half. 
In view of this fact and Lemma~\ref{lem:n squared}, we define
\[
\text{$\Pi_{\cB}: 
\{(2n)\times(2n)$ $\mathbb{R}$-valued matrices$\}\to\mathbb{R}^{n^2}$}
\]
to be the map that extracts the $n^2$ elements of a matrix $X$ selected by $\Pi_{\cA}$ that are in the top half of the matrix.

\begin{example}\label{ex:projection} 
Consider 
$c^A=
\, (1
\,\, \underline{2} 
\,\, \underline{5} \,\, \underline{6} \,\, 8 \, \, \overline{7} \,\, 
\overline{4} \,\, 
\overline{3} \,\, 
)$ and  
$c^B=
\, ({-4}
\,\, \underline{-3} 
\,\, \underline{1} \,\, \underline{2} \,\, 4 \, \, \overline{3} \,\, 
\overline{-1} \,\, 
\overline{-2} \,\, 
)$ from Example~\ref{ex:Coxeter elements A7 and B4}.
Figure~\ref{fig:ex:projection} shows the projection $\Pi_{\cA}$ and $\Pi_{\cB}$ (first two pictures). 
It also shows the permutation matrix for $\UnfoldMap(v) \in S_8$ for the $\cB$-singleton $v=(1 \ 4 \ 3 \ {-1} \ {-4} \ {-3})(2 \quad {-2})$ with circles around entries recorded by $\Pi_{\cB}$ (third picture).
\begin{figure}[htbp]
\begin{center}
\resizebox{\columnwidth}{!}{%
\begin{tabular}{|c|c|c|c|c|c|c|c|}
\hline
  & $28$&$\X$&$\X$&$24$&$19$&$\X$ & $12$ \\ \hline
  &   & $\X$  & $\X$ & $25$ & $20$ & $6$ & $13$ \\ \hline
 & & & $\X$ & $26$ & $21$ & $7$ & $14$ \\ \hline
 & & & & $27$&$22$&$8$&$15$ \\ \hline
 & & & & &$23$&$9$&$16$ \\ \hline
 &$3$& & &$\X$ &   & $10$ & $17$ \\\hline
$4$&$1$& & &$\X$&$\X$& & $18$ \\ \hline
$2$&$\X$&$5$&$11$&$\X$&$\X$& & \\ \hline
\end{tabular}
\quad
\begin{tabular}{|c|c|c|c|c|c|c|c|}
\hline
 \hspace{3mm} & $16$&$\X$&$\X$&$12$&$8$&$\X$ & $4$ \\ \hline
  &   & $\X$  & $\X$ & $13$ & $9$ & $1$ & $5$ \\ \hline
 & & & $\X$ & $14$ & $10$ & $2$ & $6$ \\ \hline
 & & & & $15$&$11$&$3$&$7$ \\ \hline
 & & & & & & & \\ \hline
 & & & &$\X$ &   &  &  \\\hline
~~ & & & &$\X$&$\X$& &  \\ \hline
 &$\X$& & &$\X$&$\X$& & \\ \hline
\end{tabular}
\quad 
\begin{tabular}{|c|c|c|c|c|c|c|c|}
\hline
 $0$ &  $\Circled{1}$& $0$&$0$&$\Circled{0}$&$\Circled{0}$&$0$ & $\Circled{0}$ \\ \hline
$0$  & $0$  & $0$  & $0$ &  $\Circled{1}$ & $\Circled{0}$ & $1$ & $\Circled{0}$ \\ \hline 
$0$ & $0$ & $0$ & $0$ & $ $$\Circled{0}$ & $\Circled{1}$ & $\Circled{0}$ & $\Circled{0}$ \\ \hline
$1$ & $0$ & $0$ & $0$ & $\Circled{0}$&$\Circled{0}$&$\Circled{0}$&$\Circled{0}$ \\ \hline
$0$ & $0$ & $0$ & $0$ & $0$ & $0$ & $0$ & $1$ \\ \hline
$0$ & $0$ & $1$ & $0$ &$0$ & $0$ & $0$ & $0$   \\\hline
$0$ & $0$ & $0$ & $1$ &$0$&$0$& $0$ & $0$  \\ \hline
$0$ &$0$& $0$ & $0$ &$0$&$0$& $1$ & $0$ \\ \hline
\end{tabular}
}

\caption{First two pictures: Projections $\Pi_{\cA}$ and $\Pi_{\cB}$ of Example \ref{ex:projection}. Red~\textcolor{red}{$X$}'s indicate the entries which must be zero by Proposition~\ref{prop:ZeroRelationsOnMatrix}. Numbers indicate entries chosen by $\Pi_{\cA}$ and $\Pi_{\cB}$, respectively, in order.  Third picture: The permutation matrix for  $\UnfoldMap(v) \in S_8$ for $v=(1 ~~ 4 ~~ 3 ~~ {-1} ~~ {-4} ~~ {-3})(2 ~~ {-2})$,
circling the entries recorded
 by $\Pi_{\cB}$.}\label{fig:ex:projection}
\end{center}
\end{figure}
\end{example}

\begin{proposition}\label{prop:PicbIsInjective}
The map $\Pi_{\cB}$ is a linear transformation which is injective on $\aff(\cB)$ and sends integral points to integral points.
\end{proposition}

\begin{proof}
Consider $X \in \aff(\cB)$. Recall that $\aff(\cB)$ is an affine subspace of $\aff(\cA)$ (see Remark \ref{rem:type B is affine subspace of type A}), so $X$ is also in $\aff(\cA)$. 
Therefore, the first $k$ positions in row $k$ for $1 \leq k \leq n$ are not guaranteed to be zero by Proposition~\ref{prop:ZeroRelationsOnMatrix} and also not chosen by $\Pi_{\cB}$.  The other entries in the row are either guaranteed to be zero by Proposition~\ref{prop:ZeroRelationsOnMatrix} or are chosen by $\Pi_{\cB}$.

Also, \cite[Theorem 4.11 ``Top sum relations'']{cBirkhoffA} tells us that there are $k$ relations involving entries of rows $1$ through $k$. Working from $k=1$ to $k=n$, we can use these relations, the values of the entries in rows above row $k$, and 
the entries of row $k$ in columns $k+1$ through $2n$
to determine the first $k$ entries of the $k$th row (see the proof of \cite[Theorem 5.9]{cBirkhoffA}).  This shows that $\Pi_{\cB}$ determines the entire top half of $X$.  By Lemma \ref{lem:X(w) is 180 rotation invariant if w is in image of iota}, we know that $X$ is invariant under 180 degree rotation, and thus the top half of~$X$ determines the bottom half of~$X$.
\end{proof}

\subsection{Proof of main theorem}
In this section, we will prove our main theorem using a composition of several maps.  The maps we will define and use throughout the section are depicted in the following commutative diagram.

\begin{equation}
\label{eq:diagram_commutes}
\begin{tikzcd}
\aff(\cA) \arrow[r,"\Pi_{\cA}"] & 
\Pi_{\cA}(\aff(\cA))\arrow[r,"\mathcal{U}_{\cA}"] &  
\mathbb{R}^{\binom{2n}{2}}\arrow[d,two heads,"P"] \\
\aff(\cB) \arrow[r,"\Pi_{\cB}"]\arrow[u,hook,"\iota"] & \Pi_{\cB}(\aff(\cB))\arrow[r,dashed, "{\mathcal{U}_{\cB}}"]\arrow[u,hook, "L"] &  
\mathbb{R}^{n^2} \\
\end{tikzcd}
\end{equation}

Recall from Remark~\ref{rem:type B is affine subspace of type A} that $\aff(\cB)$ is an affine subspace of $\aff(\cA)$, so we can define $\iota$ to be the inclusion map $\aff(\cB)\to \aff(\cA)$. 
We now define the map $L$. Since $\Pi_{\cB}$ is injective, if we restrict its codomain to be $\Pi_{\cB}(\aff(\cB))$ then it is bijective. Thus it has an inverse $\Pi_{\cB}^{-1}:\Pi_{\cB}(\aff(\cB))\to\aff(\cB)$.  Define $L=\Pi_{\cA}\circ\iota\circ\Pi_{\cB}^{-1}$. Since $L$ is a composition of injective functions, $L$ is injective.

Recall the map $f$ from Proposition~\ref{prop:poset isomorphism} and let $o(I)$ denote the indicator vector of an order ideal $I$ of a poset.  
In the proof of our main theorem in \cite{cBirkhoffA}, we showed the existence of a unimodular transformation $\mathcal{U}_{\cA}$ from $\Pi_{c^A}(\aff(c^A))$ to $\mathbb{R}^{\binom{2n}{2}}.$
By~\cite[Theorem 6.21]{cBirkhoffA}, 
\begin{align}
\label{eq:Theorem6.21}
\text{$\mathcal{U}_{\cA}\circ \Pi_{\cA}(\PermMatrix{w}) = o(f^A(w))$ for all $\cA$-singletons $w$}. 
\end{align}
Since $L$ is a map from $\Pi_\cB(\aff(\cB))$ to $\Pi_\cA(\aff(\cA))$, we can consider the composition $\mathcal{U}_{\cA} \circ L$. 

Let $\aff(J(\HHA))$ denote the affine hull of indicator vectors of order ideals of $\HHA$, and let $\aff(J(\HHA)^F)$ 
denote the affine hull of indicator vectors of
reflection-invariant order ideals of $\HHA$. 

\begin{lemma}\label{lem:ImageOfUL}
The image $\mathcal{U}_{\cA} \circ L ~ (\Pi_{\cB}(\aff(\cB)))$ is contained in $\aff(J(\HHA)^F)$.
\end{lemma}

\begin{proof}
Given a $\cA$-singleton $w$, we have $\mathcal{U}_{\cA} \circ \Pi_{\cA} (X(w)) = o(f^A(w))$ by \eqref{eq:Theorem6.21}.
Given a $\cB$-singleton $v$, we know from Lemma~\ref{lem:UnfoldMap}\eqref{lem:UnfoldMap:itm:2} that the order ideal  $f^A(\UnfoldMap(v))$ is reflection-invariant. 
Therefore, applying $\mathcal{U}_{c^A} \circ L \circ \Pi_{\cB} = \mathcal{U}_{c^A} \circ \Pi_{\cA} \circ \iota$ to the set \[
\{X(\UnfoldMap(v)) \mid v \text{ is a $\cB$-singleton} \}\] 
produces a 
set of indicator vectors of the order ideals in $J(\HHA)^F$, and the claim follows.
\end{proof}

Now we will define the map $P$.  In \cite[Section~6.1]{cBirkhoffA}, we defined a specific linear extension, $\pi_A$, of $H^A$ coming from the construction of the ``diagonal reading word''. There is an induced linear extension on $H^B$, $\pi_B$, viewing $H^B$ as a subposet of $H^A$.  
Let $P:\mathbb{R}^{\binom{2n}{2}}\to\mathbb{R}^{n^2}$ be the linear map defined by $P(\mathbf{e}_i)=\mathbf{0}$ if  $\pi_A(i)$ is labeled $k < n$ and otherwise $P(\mathbf{e}_i) = \mathbf{e}_j$ if $\pi_A(i)$ and $\pi_B(j)$ are associated to the same poset element, conflating $H^B$ with the right side of $H^A$. From this description, we see $P$ is full-rank and lattice-preserving. 

\begin{example}
Recall that Figure~\ref{fig:longest element heap A7 and B4} showed $\HHA$ for $c=\rw{7145362}$ in $A_7$ and $\HHB$ for $c=\rw{3012}$ in $B_4$. 
The linear extension $\pi_A$ is given by the following permutation in two-line notation:
\[
\setcounter{MaxMatrixCols}{28}
\scalebox{0.74}{$\begin{pmatrix}
1 & 2 & 3 & 4 & 5 & 6 & 7 & 8 & 9 & 10 & 11 & 12 & 13 & 14 & 15 & 16 & 17 & 18 & 19 & 20  & 21 & 22 & 23 & 24 & 25 & 26 & 27 & 28\\
2&3&5&7&9&4&10&12&14&16&1&6&11&17&19&21&23&8&13&18&24&26&28&15&20&25&22&27
\end{pmatrix}$
}
\]
This induces the following linear extension $\pi_B$:
\[\setcounter{MaxMatrixCols}{16}
\scalebox{0.77}{$\begin{pmatrix}
1 & 2 & 3 & 4 & 5 & 6 & 7 & 8 & 9 & 10 & 11 & 12 & 13 & 14 & 15 & 16
\\
2 &
3 & 6 &
1 & 4 & 7 & 10 & 
5 & 8 & 11 & 14 &
9 & 12 & 15 & 
13 & 16
\end{pmatrix}$}
\]
We can use these maps to compute $P:\mathbb{R}^{28}\to\mathbb{R}^{16}$.  As an example, we will compute $P(\mathbf{e}_1)$ and $P(\mathbf{e}_2)$.
Since $\pi_A(1)=2$ and $2$ has label $1<4=n$, $P(\mathbf{e}_1)=\mathbf{0}$.  
To compute $P(\mathbf{e}_2)$, first notice that $\pi_A(2)=3$ and $3$ has label $4\not<4$. Since the element~3 in $\HHA$ corresponds to the element~2 in $\HHB$ and $\pi_B(1)=2$, we have $P(\mathbf{e}_2)=\mathbf{e}_1$.
\end{example}

Define $\mathcal{U}_{\cB} := P \circ \mathcal{U}_{\cA} \circ L$.

\begin{theorem}\label{thm:Main}
The map $\mathcal{U}_{\cB} \circ \Pi_{\cB}$ is a unimodular transformation such that, for all vertices $X(\UnfoldMap(v))$ of $\bir(\cB)$, we have $\mathcal{U}_{\cB} \circ \Pi_{\cB} ~ (X(\UnfoldMap(v))) = o(f^B(v))$.
In particular, $\bir(\cB)$ is integrally equivalent to $\ord(\HHB)$.
\end{theorem}

\begin{proof}
From the commutative diagram \eqref{eq:diagram_commutes}, we see that $\mathcal{U}_{\cB} \circ \Pi_{\cB}  = P \circ \mathcal{U}_{\cA} \circ \Pi_{\cA} \circ \iota$. 
Thus, by \eqref{eq:Theorem6.21}, we have $\mathcal{U}_{\cB} \circ \Pi_{\cB} ~ (X(\UnfoldMap(v))) = P(o(f^A(\UnfoldMap(v))))$. 
By the definition of $P$,
we conclude $P(o(f^A(\UnfoldMap(v)))) = o(f^B(v))$, and thus $\mathcal{U}_{\cB} \circ \Pi_{\cB} ~ (X(\UnfoldMap(v))) = o(f^B(v))$, as desired.

The maps $L$ and $\mathcal{U}_{\cA}$ are injective.  Since $P$ preserves information from the right side of $\HHA$ and kills information from the left side of $\HHA$, $P$ is injective on 
$\aff(J(\HHA)^F)$.  
The image of $\mathcal{U}_{\cA} \circ L$ is contained in $\aff(J(\HHA)^F)$ by Lemma~\ref{lem:ImageOfUL}, so $P \circ \mathcal{U}_{\cA} \circ L=\mathcal{U}_{\cB}$ is injective.

Since $\Pi_{\cB}$ is injective on $\aff(\cB)$, 
the composition $\mathcal{U}_{\cB} \circ \Pi_{\cB}$ is also injective on $\aff(\cB)$. 
Furthermore, since $P$, $\mathcal{U}_{\cA}$, $\Pi_{\cA}$, and $\iota$ are lattice-preserving, we have that $\mathcal{U}_{\cB} \circ \Pi_{\cB}$ is also lattice-preserving. 
Therefore, $\mathcal{U}_{\cB} \circ \Pi_{\cB}$ is a unimodular transformation.
\end{proof}

\subsection*{Acknowledgements}
Banaian was supported by the German Research Foundation SFB-TRR 358/1 2023 – 491392403. Gunawan was supported by the Travel Support for Mathematicians gift SFI-MPS-TSM-00013520 from The Simons Foundation. 
Part of this work took place during the  ``Computation in Representation Theory" workshop at ICERM in November 2025. This research benefited from the open-source software {\sc SageMath}.

\printbibliography
\end{document}